\documentclass{article}
\usepackage[utf8]{inputenc}
\usepackage[english]{babel}
\usepackage{amsmath}
\usepackage{amsfonts}
\usepackage{amssymb}
\usepackage{amsthm}
\usepackage{graphicx}
\usepackage{array}
\usepackage[capitalize]{cleveref}
\usepackage{enumitem}
\usepackage{url}

\newcommand{\footremember}[2]{%
    \footnote{#2}
    \newcounter{#1}
    \setcounter{#1}{\value{footnote}}%
}
\newcommand{\footrecall}[1]{%
    \footnotemark[\value{#1}]%
} 

\newcommand{\overbar}[1]{\mkern 1.5mu\overline{\mkern-1.5mu#1\mkern-1.5mu}\mkern 1.5mu}
\newcommand{\setdelim}{\; | \;}
\newcommand{\lspace}{{L^2(X,m)}}

\newcommand{\E}{\mathcal{E}}
\newcommand{\D}{\mathfrak{D}}
\newcommand{\R}{{\mathbb R}}

\newcommand{\N}{{\mathbb N}}

\usepackage[normalem]{ulem}

\crefname{property}{Property}{Properties}
\crefname{assum}{Assumption}{Assumptions}

\theoremstyle{plain}
\newtheorem{thm}{Theorem}[section]
\newtheorem{lem}[thm]{Lemma}

\theoremstyle{definition}
\newtheorem{defn}[thm]{Definition} 

\DeclareMathOperator{\normcap}{Cap_{\mathfrak{D}}}

\DeclareMathOperator{\dom}{dom}

\begin{document}

\author{Ralph Chill \footremember{TUD}{Institut f\"ur Analysis, Fakult\"at Mathematik, Technische Universit\"at Dresden, 01062 Dresden, Germany, \texttt{ralph.chill@tu-dresden.de}, \texttt{burkhardclaus@gmx.de}} \and Burkhard Claus \footrecall{TUD}}

\date{\today}

\title{Sobolev inequalities for nonlinear Dirichlet forms}

\maketitle
 \begin{abstract}
     In this short note we show an equivalence between Sobolev type inequalities and so called isocapacitary inequalities in the context of a large class of nonlinear Dirichlet forms, their associated Dirichlet spaces and their associated capacities. 
 \end{abstract}

\section{Introduction}

Sobolev type inequalities, that is inequalities showing that some space of regular functions continuously embeds into \(L^p\), Lorentz spaces or Orlicz spaces, are one of the most basic and important tools in the study of partial differential equations. In the theory of bilinear Dirichlet forms and their associated generators of linear submarkovian \(C_0\)-semigroups, and especially for the Laplace operator and other elliptic operators, there are several well known results connecting inequalities involving the capacity to such Sobolev inequalities; see for example the significant contributions by Maz'ya \& Poborchi \cite{MP}, Maz'ya \cite{Mz11} or more recently Cianchi \& Maz'ya \cite{CiMz23}. In this short note we show that a similar connection holds for Sobolev type inequalities for Dirichlet spaces and for capacities associated with nonlinear Dirichlet forms. The notion of nonlinear Dirichlet forms was first coined in B\'enilan \& Picard \cite{BePi79a}, and further studied by Barth\'elemy \cite{By96}, Cipriani \& Grillo \cite{CG_Nonlinear_Dirichlet_Forms}, the second author in \cite{Cl21,Cl23}, Brigati \cite{Br23}, Brigati \& Hartarsky \cite{BrHa24s}, Brigati \& Dello Schiavo \cite{BrDS25}, Puchert \cite{Pu25b} and Schmidt \& Zimmermann \cite{ScZi25}. Claus also introduced the concepts of Dirichlet spaces and capacities for arbitrary nonlinear Dirichlet forms, with potential applications to $p$-Laplace operators, $p(x)$-Laplace operators, $\varphi$-Laplace operators, and general nonlinear, nonlocal operators. The Dirichlet space as well as the so-called extended Dirichlet space was further studied by Schmidt \& Zimmermann.

In the second section we recall some results on nonlinear Dirichlet forms, and their associated Dirichlet spaces and capacities. In the third section we show that certain Sobolev inequalities are equivalent to isocapacitary inequalities; we consider Sobolev inequalities for embeddings of the Dirichlet space into \(L^\infty\), \(L^{p,w}\) and \(L^p\). Finally we apply these results: imitating Stampacchia's proof, we show that solutions of nonlinear elliptic problems are bounded if the given right hand side is in \(L^p\) for some \(p\) large enough, and we prove, with the help of results from \cite{CoHa18}, that isocapacitary inequalities imply \(L^2-L^p\) inequalities for the associated nonlinear semigroup.

\section{Preliminaries on nonlinear Dirichlet forms}

We cite some definitions and results from \cite{Cl23}. All proofs are contained in this reference. In the following \((X,m)\) denotes a Hausdorff topological measure space, that is, \(X\) is a Hausdorff topological space and \(m\) is a Borel measure on \(X\). Following \cite{Cl23}, we assume that \(m\) has full support in the sense there exists no nonempty open set of measure zero. We write \({\rm supp}\, m\).

\begin{defn}
Let \(\E: \lspace \rightarrow \R \cup \{ \infty\}\) be a convex, lower semicontinuous functional. We call 
\begin{align*}
    \dom(\E)=\{ u \in \lspace \setdelim \E (u) < \infty\} 
\end{align*}
the \uline{effective domain} of \(\E\). We say that \(\E\) is \uline{proper} if its effective domain is nonempty, and \(\E\) is \uline{densely defined} if the effective domain is dense in \(\lspace\). If \(\E\) is proper, we define its \uline{subgradient} \(\partial \E\) by  
\begin{align*}
    \partial \E = \{ (u,f) \in \lspace \times \lspace \setdelim & u\in\dom{\E} \text{ and } \forall v \in \lspace : \\
    & \langle f, v -u \rangle_{L^2} \leq \E(v) - \E(u) \}.
\end{align*}
\end{defn}

\begin{thm}
Let \(\E: \lspace \rightarrow \R \cup \{ \infty\}\) be a densely defined, convex, lower semicontinuous functional. Then \(- \partial \E\) generates a strongly continuous semigroup \((T(t))_{t\geq 0}\) of contractions on \(\lspace\) in the sense that for all \(u_0 \in \lspace\) the orbit \(u = T(\cdot )u_0\in C([0,\infty );\lspace )\) is a mild solution of
\begin{align*}
    \dot{u}(t) + \partial\E(u) &\ni 0 \quad \text{in } [0,\infty ) , \\
    u(0) & = u_0 ,
\end{align*}
and in fact a strong solution on \((0,\infty )\) in the sense that  \(u\in H^1_{loc} ((0,\infty ); \lspace )\).
\end{thm}

A proof of this classical result can for example be found in Brezis \cite[Th\'eor\`emes III.3.1, III.3.2]{Br73}, Barbu \cite[Theorem 4.11, p. 78]{Bar10} or Showalter \cite[Corollary 3.3, p. 178]{Showalter_monotone_operators}.

\begin{defn} \label{def:dirichlet_form}
Let \(\mathcal{E} : \lspace \rightarrow [0,\infty]\) be a densely defined, convex, lower semicontinuous functional. We call \(\mathcal{E}\) a \uline{Dirichlet form} if
\begin{align}\label[property]{diricheqn1}
&\mathcal{E}(u \wedge v) + \mathcal{E}(u \vee v) \leq \mathcal{E}(u)+\mathcal{E}(v)
\end{align}
and
\begin{align} \label[property]{diricheqn2}
&\mathcal{E}\bigg(v+\frac{1}{2}\big( (u-v+\alpha)_+-(u-v-\alpha)_- \big)\bigg) \nonumber \\
&\quad \quad +\mathcal{E}\bigg(u-\frac{1}{2}\big( (u-v+\alpha)_+-(u-v-\alpha)_- \big)\bigg) \leq \mathcal{E}(u)+\mathcal{E}(v)
\end{align}
for every \(u,v \in \lspace,\alpha>0\).

We call \(\mathcal{E}\) \underline{even} if \(\E(0)=0\) and \(\E(-u)=\E(u)\) for all \(u \in \lspace\).
\end{defn}

Set
\begin{align*}
    \mathcal{E}_1(u)= \Vert u \Vert_\lspace^2 + \mathcal{E}(u)
\end{align*}
and define the \uline{Dirichlet space}
\begin{align*}
\mathfrak{D}=\{ u\in \lspace \setdelim \exists \lambda>0: \mathcal{E}_1(\lambda u)< \infty \}.
\end{align*}
On this space we define the Minkowski norm
\begin{align*}
\Vert u \Vert_\mathfrak{D} = \inf\Big\{ \lambda>0 : \mathcal{E}_1\Big(\frac{u}{\lambda} \Big)\leq 1 \Big\}
\end{align*}
and we let
\begin{align*}
    \vert u \vert_\mathfrak{D} = \inf\Big\{ \lambda>0 : \mathcal{E}\Big(\frac{u}{\lambda} \Big)\leq 1 \Big\}.
\end{align*}
be the associated seminorm. Although the Dirichlet space can be defined for arbitrary proper, lower semicontinuous and convex functionals \(\mathcal{E} : L^2 (X,m) \to [0,\infty ]\) with \(\mathcal{E} (0) =0\) (see \cite[Section 2 and Theorem 4.8]{Cl23}), we restrict in this article on {\em even} Dirichlet forms \(\mathcal{E}\), for example for the following property.

\begin{thm}[{\cite[Theorems 2.3 and 4.11]{Cl23}}] \label{Thm:D_Riesz_Subspace_Norm}
Let \(\mathcal{E}\) be an even Dirichlet form on \(\lspace\), and let \(\D\) be its Dirichlet space. Then \((\D, \|\cdot\|_\mathfrak{D})\) is a dual Banach space and a lattice under the pointwise order such that
\begin{align*}
    \Vert u \wedge v\Vert_\D \leq \Vert u \Vert_\D + \Vert v \Vert_\D
\end{align*}
and 
\begin{align*}
    \Vert -c \vee u \wedge c\Vert_\D \leq \Vert u \Vert_\D
\end{align*}
for all \(u, v \in \D\) and \(c\geq 0\). \label{thm:continuous_lattice_operations}
If, in addition, \(\dom \E = \D\), then the lattice operations \(\wedge,\vee\) are continuous and \(-c \vee u \wedge c\) converges to \(u\) for \(c \rightarrow \infty\) and to \(0\) for \(c \rightarrow 0\).

In addition, \(\Vert \cdot \Vert_\D\) and \(\Vert \cdot \Vert_{\lspace} + \vert \cdot \vert_\D\) are equivalent norms.
\end{thm}

\begin{thm}
Let \(\E\) be a Dirichlet form on \(\lspace\). Then the subgradient \(\partial \E\) of \(\E\) is a completely $m$-accretive operator as defined in \cite{BC_completly_accretive_operators}. In particular, the semigroup on $\lspace$ generated by the negative subgradient \(-\partial\E\) leaves \(L^2\cap L^p (X,m)\) invariant and extends from there to a strongly continuous semigroup of contractions on \(L^p (X,m)\) for every \(p\in [1,\infty)\). 
\end{thm}

\begin{defn}
Let \(\E : \lspace \to [0,\infty ]\) be a Dirichlet form and let \(A \subseteq X\). We define the set
\begin{align*}
    \mathcal{L}_A=\{u \in \lspace \setdelim u\geq 1 \text{ on } U, A \subseteq U,\; U \text{ open} \},
\end{align*}
and the \underline{norm-capacity} by
\[
\normcap(A)=\inf\{\Vert u \Vert_\D \setdelim u \in \mathcal{L}_A\} .
\]

We call a set \(A\subseteq X\) \underline{polar}, if \( \normcap(A)=0\) and some property depending on \(x\in X\) holds \uline{quasi everywhere} (q.e.), if there is a polar set \(A\) such that the property holds for every \(x \in X \setminus A\).
 We say that a function \(f: X  \rightarrow Y\) with values in some topological space \(Y\) is \uline{quasi-continuous}, if for every \(\varepsilon > 0 \) there is an open set \(O\subseteq X \) such that \(\normcap(O) \leq \varepsilon \) and \(f|_{O^c}\) is continuous.

Additionally we call \(f \in \lspace\) \uline{quasi-continuous}, if there is a representative of \(f\) which is quasi-continuous. Whenever this is the case, we denote this representative again by \(f\).
\end{defn}

\begin{thm}[{\cite[Corollary 6.7]{Cl23}}] \label{thm:cheb_inequality}
Let \(f\in \D\) be quasi-continuous. Then
\begin{align*}
    \normcap(\{\vert f \vert \geq \lambda \}) \leq \frac{\Vert f \Vert_{\D}}{\lambda}.
\end{align*}
\end{thm}

\begin{thm}[{\cite[Corollary 6.8]{Cl23}}] \label{thm:D_convergnes_implies_pointwise_quasi_everywhere}
Let \((f_n)_n\) be a sequence of quasi-continuous functions in \(\mathfrak{D}\) and \(f \in \mathfrak{D}\) such that \(f_n \rightarrow f \) in \( \mathfrak{D}\). Then \(f\) is quasi-continuous and there exists a subsequence which converges pointwise quasi everywhere.
\end{thm}

\begin{thm}[{\cite[Corollary 6.9]{Cl23}}]
Let \(f \in \overbar{\mathfrak{D} \cap C(X)}^{\D}\). Then \(f\) is quasi-continuous on \(X\).
\end{thm}

\section{Isocapacitary inequalities}

Let \((X,m)\) be a Hausdorff topological measure space such that \({\rm supp} (m) =X\). 

\begin{lem}
Let \(\E\) be an even Dirichlet form on \(\lspace\), and let \(\D\) be its Dirichlet space. Then, the embedding \(\D \hookrightarrow \lspace\) is compact if \(\{f \in \lspace \setdelim \E (f) \leq 1\}\) is (relatively) compact in \(\lspace\), and the embedding is compact if and only if \(\{f \in \lspace \setdelim \E_1(f) \leq 1 \}\) is (relatively) compact in \(\lspace\).
\end{lem}

\begin{proof} 
The unit ball in \(\D\) equals \(\{f \in \lspace \setdelim \E_1 (f) \leq 1\}\), and
\begin{align*}
\{f \in \lspace \setdelim \E_1(f) \leq 1 \} \subseteq \{f \in \lspace \setdelim \E(f) \leq 1 \} .
\end{align*}
As the functions \(\E_1\) and \(\E\) are lower semicontinuous, both the sets in question are closed in \(\lspace\), and hence they are relatively compact if and only they are compact. 
\end{proof}

\begin{defn}
Let \(\E\) be an even Dirichlet form on \(\lspace\). We say the Dirichlet space \(\D\) satisfies a \uline{Poincaré inequality} \index{Poincaré inequality} if there is a constant \(C>0\) such that
\begin{align*}
    \Vert f \Vert_\D \leq C \vert f \vert_\D \quad \text{for all } f \in \D,
\end{align*}
or, equivalently, 
\begin{align*}
    \Vert f \Vert_\lspace \leq C \vert f \vert_\D \quad \text{for all } f \in \D.
\end{align*}
\end{defn}

\begin{thm}
Let \(\E\) be an even Dirichlet form on \(\lspace\), and let \(\D\) be its Dirichlet space.  Assume that the embedding \(\D \hookrightarrow \lspace\) is compact and that \(\E(f)=0\) if and only if \(f=0\). Then, \(\D\) satisfies a Poincar\'e inequality.
\end{thm}

\begin{proof}
Let \(K=\{f \in \D \setdelim \Vert f \Vert_\D \leq 2\}\). Then, the assumption implies that \(K\) is compact in \(\lspace\). Hence, \(A= K \cap \{ f \in \lspace \setdelim \Vert f \Vert_\lspace = 1\}\) is compact in \(\lspace\). We define
\begin{align*}
    C= \inf_{f \in A} \vert f \vert_\D.
\end{align*}
By definition of the infimum, there is a sequence \((f_n)_{n \in \N}\) in \(A\) such that
\begin{align*}
    C = \lim_{n \rightarrow \infty} \vert f_n \vert_\D.
\end{align*}
Since \(A\) is compact there is a function \(f \in A\) and a subsequence of \((f_n)_{n \in \N}\), again denoted by \((f_n)_{n \in \N}\), such that \(f_n \rightarrow f\) in \(\lspace\). Since \(\vert \cdot \vert_\D\) is lower semicontinuous on \(\lspace\),
\begin{align*}
   \vert f \vert_\D \leq C .
\end{align*}
Note that \(f \neq 0\), since \(f \in A\). Hence,
\begin{align*}
    C > 0.
\end{align*}
Therefore,
\begin{align*}
    \Vert u \Vert_\lspace = 1 \leq \frac{1}{C} \vert u \vert_\D
\end{align*}
for every \(u \in A\). By definition of \(A\), this inequality holds for every \(u \in \{ f \in \lspace \setdelim \Vert f \Vert_\lspace = 1\}\). A scaling argument yields the claim.
\end{proof}

\begin{thm}[Embedding into \(L^\infty\)] \label{thm:isocapa_linfty} \index{Sobolev inequality}
Let \(\E\) be an even Dirichlet form on \(\lspace\), and let \(\D\) be its Dirichlet space. Assume that every \(f \in \D\) is quasi-continuous, and let \(C >0\). Then, the following are equivalent:  \renewcommand{\labelenumi}{(\roman{enumi})}
\begin{enumerate}
    \item Every \(f \in \D\) is continuous and bounded, and \(\Vert f \Vert_{L^\infty (X,m)}  \leq C \Vert f \Vert_\D\).
    \item For every \(x \in X\), \(\normcap(\{x\}) \geq \frac{1}{C}\).
\end{enumerate}
\end{thm}

\begin{proof}
Let us assume (i). Let \(x \in X\) and let \(U \subseteq    X\) be an open neighbourhood of \(x\). Let \(f \in \mathcal{L}_U\). Then, \(f \geq 1\) on \(U\). Hence, 
\begin{align*}
    \Vert f \Vert_\D \geq \frac{1}{C} \Vert f \Vert_{L^\infty (X,m)}  \geq \frac{1}{C}.
\end{align*}
Thus, 
\begin{align*}
    \normcap(\{x\}) \geq \frac{1}{C}.
\end{align*}
For the converse implication (ii)\(\Rightarrow\)(i), let us choose an arbitrary \(f \in \D\). Let \(\varepsilon < \frac{1}{C}\). Since \(f\) is quasi-continuous, there is an open set \(U \subseteq    X\) such that \(\normcap(U) \leq \varepsilon\) and \(f\) is continuous on \(U^c\). Since \(\normcap\) is a monotone set function, \(U\) is either the empty set or
\begin{align*}
    \frac{1}{C} \leq \normcap(U) \leq \varepsilon < \frac{1}{C},
\end{align*}
which is a contradiction. Hence, \(U =\emptyset\). As a consequence, \(f \) is continuous. Now, let us assume there is a point \(x \in X\) such that \(\vert f(x) \vert \geq  C \Vert f \Vert_\D  + \varepsilon\) for some \(\varepsilon > 0\). Then, by \cref{thm:cheb_inequality},
\begin{align*}
    \frac{1}{C} &\leq \normcap(\{x\}) \\
    &\leq \normcap\big(\{ \vert f \vert \geq  C \Vert f \Vert_\D  + \varepsilon\}\big) \\
    &\leq \frac{ \Vert f \Vert_\D}{C \Vert f \Vert_\D  + \varepsilon} \\
    &< \frac{1}{C}.
\end{align*}
This is a contradiction. Thus, there is no such \(x\) and \(\Vert f \Vert_{L^\infty (X,m)}  \leq C \Vert f \Vert_\D\).
\end{proof}

\begin{defn}
Let \(f\) be a measurable function on \((X,m)\), \(p\geq 1\) and define the \uline{distribution function} \(m_f(\lambda) := m\big(\{ \vert f \vert \geq \lambda \}\big)\). We define the quasinorm
\begin{align*}
    \Vert f \Vert_{L^{p,w}(X,m)} = \sup_{\lambda \geq 0} \lambda \, m_f(\lambda)^{\frac{1}{p}}
\end{align*}
and the \uline{weak \(L^p\) space} by
\begin{align*}
    L^{p,w}(X,m)= \{ f: X \rightarrow \R \setdelim f \text{ measurable and } \Vert f \Vert_{L^{p,w}(X,m)} < \infty \}.
\end{align*}
Some authors denote this space by \(L^{p,\infty}(X,m)\). It is a special space in the scale of Lorentz spaces.
\end{defn}

\begin{thm}[Embedding into \(L^{p,w}\)]
Let \(\E\) be an even Dirichlet form on \(\lspace\), and let \(\D\) be its Dirichlet space.  Assume that every \(f \in \D\) is quasi-continuous, let \(p>2\) and let \(C \geq 0\). Then, the following are equivalent: \renewcommand{\labelenumi}{(\roman{enumi})}
\begin{enumerate}
    \item \(  \Vert f \Vert_{L^{p,w}(X,m)} \leq C \Vert f \Vert_\D\) for every \(f \in \D\).
    \item \( m(A) \leq C^p \normcap(A)^p \) for every measurable \(A \subseteq X\).
\end{enumerate}
\end{thm}

\begin{proof}
Let us assume (i). By \cref{thm:cheb_inequality} and the assumptions,
\begin{align*}
    m_f(\lambda ) = m\big(\{ f \geq \lambda \}\big) \leq C^p \normcap(\{ \vert f \vert \geq \lambda \})^p \leq C^p \frac{ \Vert f \Vert_\D^p }{\lambda^p}.
\end{align*}
Hence,
\begin{align*}
    \Vert f \Vert_{L^{p,w}(X,m)} \leq C \Vert f \Vert_\D.
\end{align*}

For the converse implication (ii)\(\Rightarrow\)(i), let \(A \subseteq X\) and \(U\) be an open neighbourhood of \(A\). Let \(f \in \mathcal{L}_U\). Then, \(f \geq 1\) on \(A\). Thus,
\begin{align*}
    m(A) \leq m_f(1) = 1^p m_f(1) \leq \Vert f \Vert_{L^{p,w}(X,m)}^p \leq C^p \Vert f \Vert_\D^p.
\end{align*}
Since \(U\) and \(f \in \mathcal{L}_U\) are arbitrary,
\begin{align*}
    m(A) \leq  C^p \normcap(A)^p.
\end{align*}
\end{proof}

Finally, we show a similar embedding result for the embedding into classical \(L^p\) spaces.

\begin{thm}[Embedding into \(L^q\)] \label{thm:isocapa_lp} 
Let \(\E\) be an even Dirichlet form on \(\lspace\), and let \(\D\) be its Dirichlet space. Assume that every \(f \in \D\) is quasi-continuous and let \(p>2\). Then, the following are equivalent:
\renewcommand{\labelenumi}{(\roman{enumi})}
\begin{enumerate}
    \item For every \(q \in [2,p) \) there exists a \(C>0\) such that \(\Vert f \Vert_{L^q} \leq C \Vert f \Vert_{\D}\) for every \(f \in \D\).
    \item For every \(q \in [2,p)\) there exists a \(C > 0\) such that \(C \normcap(A) \geq m(A)^{\frac{1}{q}}\) for every measurable subset \(A \subseteq    X\).
\end{enumerate}
\end{thm}

\begin{proof}
Let \(q \in [2,p) \) be fixed and let us assume (i). Let \(A\subseteq    X\) be measurable, \(U\) an open neighbourhood of \(A\) and \(u \in \mathcal{L}_U\). Then,
\begin{align*}
    \Vert u \Vert_\D \geq \frac{1}{C} \Vert u \Vert_{L^q(X,m)} \geq \frac{1}{C} \Vert 1_U \Vert_{L^q(X,m)} = \frac{1}{C} m(A)^{\frac{1}{q}}.
\end{align*}
Thus, since \(U\) and \(u \in \mathcal{L}_U\) are arbitrary,
\begin{align*}
    m(A)^{\frac{1}{q}} \leq C \normcap(A). 
\end{align*}
For the converse direction (ii)\(\Rightarrow\)(i), let \(f \in \D\) and \(q<p\). Then,
\begin{align*}
    \int_X  \vert f \vert^q dm = \int_{\vert f \vert < 1} \vert f \vert^q dm + \int_{\vert f \vert \geq 1} \vert f \vert^q dm
\end{align*}
and
\begin{align*}
    \int_{\vert f \vert < 1} \vert f \vert^q dm \leq \int_X \vert f \vert^2 dm \leq \Vert f \Vert_\D^2.
\end{align*}
To estimate the second part, let \(\varepsilon >0\) such that \(q_\varepsilon=q+\varepsilon < p\). Then,
\begin{align*}
     \int_{\vert f \vert \geq 1} \vert f \vert^q dm &= \int_1^\infty m\big( \{\vert f \vert \geq \lambda\}\big) d(\lambda^q) \\
    &=q \int_1^\infty m\big( \{\vert f \vert \geq \lambda\}\big)  \lambda^{q-1} d\lambda \\
    &\leq q \int_1^\infty C^{q_\varepsilon} \normcap\big(\{\vert f \vert \geq \lambda\}\big)^{q_\varepsilon} \lambda^{q-1} d\lambda \\
    & \leq q \int_1^\infty C^{q_\varepsilon} \frac{1}{\lambda^{q_\varepsilon}} \Vert f \Vert_\D^{q_\varepsilon} \lambda^{q-1} d\lambda \\
    &= q C^{q_\varepsilon} \Vert f \Vert_\D^{q_\varepsilon} \int_1^\infty \lambda^{q-1-{q_\varepsilon}} d\lambda.
\end{align*}
This integral on the right hand side is finite. Hence, there is a constant \(C_1\) only depending on \(q,\varepsilon\) and \(C\) such that
\begin{align*}
    \int_X  \vert f \vert^q dm \leq C_1 \big( \Vert f \Vert_\D^{q_\varepsilon} +\Vert f \Vert_\D^2 \big).
\end{align*}
Thus, \(f\in L^q\). Note that this inequality also implies that the embedding \( \D \rightarrow L^q(X)\) is continuous in \(0\), thus bounded. This yields the claim.
\end{proof}

\section{Elliptic Regularity}

\begin{defn}
Let \(\E\) be a Dirichlet form and \(r \geq 2\). We say \(\E\) has \uline{growth type at most \(r\)} \index{growth type} if 
\begin{align*}
    \E(\lambda u) \leq \lambda^r \E(u)
\end{align*}
for every \(u \in \lspace\) and \(\lambda\geq 1\).
\end{defn}

\begin{lem}\label{thm:grwoth_type}
Let \(\E\) be an even Dirichlet form with growth type at most \(r\geq 2\). Then,
\begin{align*}
    \Vert u \Vert_\D ^r \leq \E_1(u)
\end{align*}
for every \(u \in \D\) such that \(\Vert u \Vert_\D \leq 1\).
\end{lem}

\begin{proof}
Obviously, \(\E_1\) has also growth type at most \(r\).
Let \(u \in \D\) such that \(\Vert u \Vert_\D \leq 1\). Then, 
\begin{align*}
    \frac{1}{\Vert u \Vert_\D} \geq 1
\end{align*}
and
\begin{align*}
    1= \E_1 \left( \frac{u}{\Vert u \Vert_\D} \right) \leq \frac{1}{\Vert u \Vert_\D^r} \E_1(u).
\end{align*}
This implies
\begin{align*}
    \Vert u \Vert_\D ^r \leq \E_1(u).
\end{align*}
\end{proof}

\begin{thm}
Let \(\E\) be an even Dirichlet form with growth type at most \(r\), \(q \geq 2\), \(\partial \E\) the subgradient of \(\E\), \(\lambda > 0 \), \(f \in L^q(X,m)\) and \(u\) the solution of
\begin{align*}
    \partial \E (u)+\lambda u =f.
\end{align*}
Let \(p\geq 2\) such that \({\left(1-\frac{2}{q}\right) \frac{p}{r}}> 1\) and let us assume that \(\D\) embeds continuously into \(L^p(X,m)\). Then, \(u \in L^\infty(X,m)\).
\end{thm}

\begin{proof}
Let us first inspect the case \(\lambda=2\). Let \(u \in \lspace\) be the solution of 
\begin{align*}
    \partial \E (u)+ 2u =f
\end{align*}
and set \(M(r)= \{ f>r\}\).
Let \(u_k= (-k) \vee u \wedge k\) and \(\zeta_k= u -u_k\). The functional \(\E_1\) is, as before, given by
\begin{align*}
    \E_1(g)=\E(g) + \Vert g \Vert_\lspace^2 
\end{align*}
for every \(g \in \lspace\). Note that \(\E_1\) is an even Dirichlet form. Furthermore 
\begin{align*}
    \partial \E_1 = \partial \E +  2\text{Id}.
\end{align*}
Hence, by the definition of \(u\) and the subgradient,
\begin{align*}
    \E_1(u)-\E_1(u_k) \leq \langle u - u_k, f \rangle_{L^2} = \langle \zeta_k , f \rangle_{L^2}.
\end{align*}
In addition, since \(p(x) = (-k) \vee x \wedge k\) is a normal contraction, the Beurling-Deny criteria imply
\begin{align*}
    \E_1(u-u_k)+\E_1(u_k)\leq \E_1(u)+\E_1(0)=\E_1(u).
\end{align*}
Hence,
\begin{align*}
    \E_1(\zeta_k)\leq \langle \zeta_k , f \rangle_{L^2}.
\end{align*}
This implies
\begin{align*}
    \E_1(\zeta_k)&\leq \langle \zeta_k , f \rangle_{L^2} = \int_{M(k)} \zeta_k f \\
    &\leq \bigg( \int_{M(k)} \zeta_k^2 \bigg) ^{\frac{1}{2}} \bigg( \int_{M(k)} f^2 \bigg) ^{\frac{1}{2}} \\ &\leq \frac{\Vert \zeta_k \Vert_\lspace^2}{2}+\frac{1}{2} \int_{M(k)} f^2,
\end{align*}
where we used the Cauchy-Schwarz inequality and Young's inequality, respectively.
Subtracting the \(\lspace\) norm yields
\begin{align*}
    \frac{1}{2} \E_1 ( \zeta_k) \leq \E(\zeta_k)+\frac{1}{2} \Vert  \zeta_k \Vert_\lspace &\leq \frac{1}{2} \int_{M(k)} f^2.
\end{align*}
Now, since \(\E\) has growth type at most \(r\), the functional \(\E\) is quasilinear. Hence, by \cref{thm:continuous_lattice_operations}, \(\zeta_k \rightarrow 0\). As a consequence, we can choose a \(k \in \N\) large enough such that \(\vert \zeta_k \vert_\D \leq 1\).
Thus, by \cref{thm:grwoth_type} and Hölder's inequality,
\begin{align*}
    \frac{1}{2} \Vert \zeta_k \Vert_\D^r \leq \frac{1}{2} \E_1(\zeta_k) \leq \frac{1}{2} \int_{M(k)} f^2 \leq m\big(M(k)\big)^{1-\frac{q}{2}} \Vert f \Vert_{L^q(X,m)}^2.
\end{align*}
Let \(h \geq k\). Then, there exists a constant \(c \geq 0\) such that
\begin{align*}
    \Vert \zeta_k \Vert_\D^r &\geq c \Vert \zeta_k \Vert_{L^p(X,m)}^r \\& \geq c \left( \int_{M(k)} (\vert u \vert -k)^p \right)^\frac{r}{p} \\ &\geq c \left( \int_{M(h)} (\vert u \vert -k)^p \right)^{\frac{r}{p}} \\&\geq c (h-k)^r m\big(M(h)\big)^{\frac{r}{p}}.
\end{align*}
The previous two inequalities together yield
\begin{align*}
    \frac{c}{2} (h-k)^r m\big(M(h)\big)^\frac{r}{p} \leq m\big(M(k)\big)^{1-\frac{2}{q}} \Vert f \Vert_{L^q}^2
\end{align*}
or, in other words,
\begin{align*}
    m\big(M(h)\big) \leq \frac{\hat{C}}{(h-k)^p}  m\big(M(k)\big)^{\left(1-\frac{2}{q}\right) \frac{p}{r}} \Vert f \Vert_{L^q}^\frac{p}{r}.
\end{align*}
Hence, by Stampacchia's Lemma \cite[Chapter 2 Lemma B.1]{SK_Stampacchias_Lemma}, there is a \(k_0 \in \R^+\) such that \(m(M(k_0))=0\), and therefore \(u \in L^\infty(X,m)\), if 
\begin{align*}
    {\left(1-\frac{2}{q}\right) \frac{p}{r}}> 1.
\end{align*}
Thus, under these assumptions on \(q,p,r\), the resolvent \(J_2\) of \(\partial \E\) maps \(L^q(X,m)\) into \( L^\infty(X,m)\). 

Let us assume
\begin{align*}
    {\left(1-\frac{2}{q}\right) \frac{p}{r}}> 1 ,
\end{align*}
and inspect the case \(\lambda \neq 2\). By \cite{BC_completly_accretive_operators}, the resolvent \(J_\lambda\) of \(\E\) is a contraction on \(L^q(X,m)\). The resolvent identity states that
\begin{align*}
    J_\lambda = J_2\left( \frac{2}{\lambda} \text{Id} + \left( 1 - \frac{2}{\lambda}\right) J_\lambda \right).
\end{align*}
Note that \(J_\lambda\) maps \(L^q(X,m)\) to \(L^q(X,m)\) and \(J_2\) maps \(L^q(X,m)\) to \( L^\infty(X,m)\). Hence, \(J_\lambda\) maps \(L^q(X,m)\) to \(L^\infty(X,m)\). 
\end{proof}

\section{\(L^2\)-\(L^p\)-regularization of semigroups}

Sobolev embeddings of Dirichlet spaces of Dirichlet forms into \(L^p (X,m)\) may imply a \(L^2\)-\(L^p\)-regularization result of the associated semigroups. The isocapacitary inequalities are not directly involved here, but as shown here, they imply Sobolev embeddings. We state a corollary for special Dirichlet forms. 

\begin{thm}
Let $\E$ be an even Dirichlet form that satisfies the following variant of the Poincar\'e inequality: there exists $\sigma >0$ and $C_1\geq 0$ such that
\[
\| u\|_{\mathfrak{D}}^\sigma \leq C_1 \, \E (u) \text{ for every } u \in\dom{\E} .
\]
Assume in addition that \(\mathfrak{D}\) satisfies a Sobolev type embedding into \(L^p (X ,m)\) for some $p \geq 2$, that is, there exists $C_2 \geq 0$ such that $\|u\|_{L^p(X, m)} \leq C_2 \, \|u\|_{\mathfrak{D}}$ for every $u \in \mathfrak{D}$. Let \((T_t)_{t\geq 0}\) be the semigroup generated by \(-\partial\E\). Then, there is a constant $K>0$ such that
\[
\left\|T_t u\right\|_{L^p(X, m)} \leq K \frac{1}{t^\sigma}\, \|u\|_{L^2(X, m)}^{\frac{2}{\sigma}}
\]
for every $u \in L^2(X, m)$ and every $t \geq 0$.
\end{thm}

\begin{proof}
    By assumption, for every \(u\in \dom{E}\),
\[
\| u\|_{L^p(X, m)}^\sigma \leq C_2^\sigma \, \|u\|_{\mathfrak{D}}^\sigma \leq C_1 C_2^\sigma \, \E (u) .
\]
By definition of the subgradient, as \(0\) is a minimizer of \(\E\) (so that also \((0,0)\in\partial\E\)), and as \(\E (0) = 0\), for every \((u,f)\in\partial\E\),
\[
\E (u) \leq \langle f , u\rangle_{L^2} .
\]
Taking the preceding two inequalities together, we find that the subgradient of \(\E\) satisfies a special case of a Sobolev type inequality (\cite[Definition 3.2]{CoHa18}), and therefore also a special case of a Gagliardo-Nirenberg inequality (\cite[Definition 1.1]{CoHa18}): for every \((u,f)\in\partial\E\)
\[
\| u\|_{L^p(X, m)}^\sigma \leq C_1C_2^\sigma \, \langle f , u\rangle_{L^2} .
\]
The claim follows from \cite[Theorem 3.8]{CoHa18}.
\end{proof}

{\bf Conflict of Interest.}  On behalf of all authors, the corresponding author states that there is no conflict of interest.


\begin{thebibliography}{BDS25}

\bibitem[Bar96]{By96}
L.~Barth\'elemy.
\newblock Invariance d'un convex ferm\'e par un semi-groupe associ\'e \`a une
  forme non-lin\'eaire.
\newblock {\em Abst. Appl. Anal.}, 1:237--262, 1996.

\bibitem[Bar10]{Bar10}
Viorel Barbu.
\newblock {\em Nonlinear differential equations of monotone types in {B}anach
  spaces}.
\newblock Springer Monographs in Mathematics. Springer, New York, 2010.

\bibitem[BC91]{BC_completly_accretive_operators}
Ph. B\'{e}nilan and M.~G. Crandall.
\newblock Completely accretive operators.
\newblock In {\em Semigroup theory and evolution equations ({D}elft, 1989)},
  volume 135 of {\em Lecture Notes in Pure and Appl. Math.}, pages 41--75.
  Dekker, New York, 1991.

\bibitem[BDS25]{BrDS25}
G.~Brigati and L.~Dello~Schiavo.
\newblock Non-bilinear {D}irichlet functionals: Markovianity, locality,
  invariance.
\newblock \url{https://doi.org/10.48550/arXiv.2510.02942}, 2025.

\bibitem[BH24]{BrHa24s}
Giovanni Brigati and Ivailo Hartarsky.
\newblock The normal contraction property for non-bilinear {D}irichlet forms.
\newblock {\em Potential Anal.}, 60(1):473--488, 2024.

\bibitem[BP79]{BePi79a}
Philippe B{\'e}nilan and Colette Picard.
\newblock Quelques aspects non lin\'eaires du principe du maximum.
\newblock In {\em S\'eminaire de {T}h\'eorie du {P}otentiel, {N}o. 4 ({P}aris,
  1977/1978)}, volume 713 of {\em Lecture Notes in Math.}, pages 1--37.
  Springer, Berlin, 1979.

\bibitem[Bre73]{Br73}
H.~Brezis.
\newblock {\em Op\'erateurs maximaux monotones et semi-groupes de contractions
  dans les espaces de {H}ilbert}, volume~5 of {\em North Holland Mathematics
  Studies}.
\newblock North-Holland, Amsterdam, London, 1973.

\bibitem[Bri23]{Br23}
G.~Brigati.
\newblock Nonlinear {D}irichlet forms, energy spaces, and calculus rules.
\newblock \url{https://doi.org/10.48550/arXiv.2309.00377}, 2023.

\bibitem[CG03]{CG_Nonlinear_Dirichlet_Forms}
F.~Cipriani and G.~Grillo.
\newblock Nonlinear {M}arkov semigroups, nonlinear {D}irichlet forms and
  applications to minimal surfaces.
\newblock {\em Journal für die reine und angewandte Mathematik (Crelle's
  Journal)}, 562:201--235, 2003.

\bibitem[CH16]{CoHa18}
Thierry Coulhon and Daniel Hauer.
\newblock Functional inequalities and regularizing effect of nonlinear
  semigroups.
\newblock \url{https://arxiv.org/abs/1604.08737}, 2016.

\bibitem[Cla21]{Cl21}
Burkhard Claus.
\newblock {\em Non-linear {D}irichlet forms}.
\newblock PhD thesis, Technische Universit\"at Dresden, Dresden, 2021.

\bibitem[Cla23]{Cl23}
B.~Claus.
\newblock Energy spaces, {D}irichlet forms and capacities in a nonlinear
  setting.
\newblock {\em Potential Anal.}, 58(1):159--179, 2023.

\bibitem[CM23]{CiMz23}
Andrea Cianchi and Vladimir~G. Maz'ya.
\newblock Sobolev embeddings into {O}rlicz spaces and isocapacitary
  inequalities.
\newblock {\em Trans. Amer. Math. Soc.}, 376(1):91--121, 2023.

\bibitem[KS00]{SK_Stampacchias_Lemma}
D.~Kinderlehrer and G.~Stampacchia.
\newblock {\em An Introduction to Variational Inequalities and Their
  Applications}.
\newblock Society for Industrial and Applied Mathematics, 2000.

\bibitem[Maz11]{Mz11}
Vladimir Maz'ya.
\newblock {\em Sobolev spaces with applications to elliptic partial
  differential equations}, volume 342 of {\em Grundlehren der Mathematischen
  Wissenschaften [Fundamental Principles of Mathematical Sciences]}.
\newblock Springer, Heidelberg, augmented edition, 2011.

\bibitem[MP97]{MP}
V.G. Maz{\cprime}ya and S.V. Poborchi.
\newblock {\em Differentiable functions on bad domains}.
\newblock World Scientific Publishing Co., Inc., River Edge, NJ, 1997.

\bibitem[Puc25]{Pu25b}
Simon Puchert.
\newblock Nonlinear {B}eurling-{D}eny criteria.
\newblock \url{https://doi.org/10.48550/arXiv.2502.03691}, 2025.

\bibitem[Sho97]{Showalter_monotone_operators}
R.~E. Showalter.
\newblock {\em Monotone operators in {B}anach space and nonlinear partial
  differential equations}, volume~49 of {\em Mathematical Surveys and
  Monographs}.
\newblock American Mathematical Society, Providence, RI, 1997.

\bibitem[SZ25]{ScZi25}
Marcel Schmidt and Ian Zimmermann.
\newblock The extended {D}irichlet space and criticality theory for nonlinear
  {D}irichlet forms.
\newblock \url{https://doi.org/10.48550/arXiv.2501.18391}, 2025.

\end{thebibliography}
\bibliographystyle{alpha}

\def\ocirc#1{\ifmmode\setbox0=\hbox{$#1$}\dimen0=\ht0 \advance\dimen0
  by1pt\rlap{\hbox to\wd0{\hss\raise\dimen0
  \hbox{\hskip.2em$\scriptscriptstyle\circ$}\hss}}#1\else {\accent"17 #1}\fi}
  \def\cprime{$'$} \def\ocirc#1{\ifmmode\setbox0=\hbox{$#1$}\dimen0=\ht0
  \advance\dimen0 by1pt\rlap{\hbox to\wd0{\hss\raise\dimen0
  \hbox{\hskip.2em$\scriptscriptstyle\circ$}\hss}}#1\else {\accent"17 #1}\fi}

\end{document}